\documentclass{amsart}
\usepackage{amssymb,amsmath,amsthm}
\usepackage[shortlabels]{enumitem}
\usepackage{comment}

\newtheorem{theorem}{Theorem}[section]

\newtheorem{lemma}[theorem]{Lemma}
\newtheorem{proposition}[theorem]{Proposition}
\newtheorem{definition}[theorem]{Definition}

\newtheorem{fact}[theorem]{Fact}

\theoremstyle{definition}
\newtheorem{remark}[theorem]{Remark}

\newcommand{\N}{\mathbb{N}}
\newcommand{\Z}{\mathbb{Z}}

\newcommand{\T}{\mathbb{T}}
\newcommand{\C}{\mathbb{C}}

\newcommand{\cU}{\mathcal{U}}

\newcommand{\ee}{\varepsilon}

\title{Selfless reduced amalgamated free products and HNN extensions}

\author[D. Gao]{David Gao}
\address{Department of Mathematical Sciences, University of California, San Diego, 9500 Gilman Dr, La Jolla, CA 92092, USA}\email{weg002@ucsd.edu}\urladdr{https://sites.google.com/ucsd.edu/david-gao}

\author[S. Kunnawalkam Elayavalli]{Srivatsav Kunnawalkam Elayavalli}
\address{\parbox{\linewidth}{Department of Mathematics, University of Maryland, College Park, \\
		4176 Campus Dr, College Park, MD 20742}}
\email{sriva@umd.edu}
\urladdr{https://sites.google.com/view/srivatsavke}

\author[G. Patchell]{Gregory Patchell}
\address{\parbox{\linewidth}{Mathematical Institute, University of Oxford, Andrew Wiles Building, \\ Radcliffe Observatory Quarter, Woodstock Road, Oxford, OX2 6GG, UK}}
\email{greg.patchell@maths.ox.ac.uk}
\urladdr{https://sites.google.com/view/gpatchel}

\author[L. Teryoshin]{Lizzy Teryoshin}
\address{Department of Mathematical Sciences, University of California, San Diego, 9500 Gilman Drive \#0112, La Jolla, CA 92093, USA}
\email{eteryoshin@ucsd.edu}

\begin{document}

\begin{abstract}
We find a general family of selfless inclusions in reduced amalgamated free products of $C^*$-algebras. Apart from generalizing prior works due to McClanahan, Ivanov and Omland, our work yields a few other applications. We present a short new approach to construct HNN extensions of $C^*$-algebras and find several new examples of selflessness using this. This generalizes results of Ueda, Ivanov and de la Harpe--Pr\'{e}aux. As another application our work yields a short proof of selflessness for arbitrary graph products of $C^*$-algebras over graphs of more than 2 vertices and diameter greater than 3.

\end{abstract}
	\maketitle

\section{Introduction}

The amalgamated free product construction is of central importance in modern combinatorial and geometric group theory. Together with the Higman--Neumann--Neumann (HNN) extension construction \cite{HNNpaper}, these form the two pillars of Bass--Serre theory, which fully clarifies the structure of groups acting by isometries on trees \cite{Ser80}. Several natural families of groups are built from these constructions, including relatively hyperbolic groups, several linear groups, 3-manifold groups, Artin/Coxeter groups, acylindrically hyperbolic groups \cite{MO15}, etc. These constructions have also played an important role in fundamental embedding theorems in combinatorial group theory \cite{higman1961subgroups}. 

 In the last 50 years, the free product construction and its generalizations have gained significant prominence in the study of operator algebras \cite{voiculescu1992free}. Importantly, Voiculescu defined amalgamated free products naturally via operator valued Fock-space representations \cite{voiculescu1985symmetries}. This construction produces relative freeness phenomena with the flexibility of allowing additional constraints. Over the years, this has proved to be of outstanding utility: in free probability theory \cite{speicher1998combinatorial}, subfactor theory \cite{Popa1993}, deformation/rigidity theory \cite{IPP08, CartanAFP}, non-commutative $L^p$-space theory \cite{KhinRX, mei}, continuous model theory \cite{exoticCIKE}, etc. More recently, Ueda also introduced HNN extensions in operator algebras \cite{Ueda2005HNN, Ueda2008Remarks} and established basic properties including universality theorems. Despite the fact that HNN extensions are proved therein to be compressions of certain amalgamated free products, there is independent value to study such algebras \cite{fima2012hnn, gao2025conjugacy}. 

Within the context of $C^*$-algebras, the foundational works \cite{Avitzour, Dykemasimplicitystablerankonefree, stablerankHaagerup, dykema1998projections}  studied simplicity, monotraciality, strict comparison and stable rank one among certain families of reduced free products. The study of simplicity in reduced amalgamated free products was first carried out in the work of McClanahan \cite{McClanahan1994}, where several examples were obtained. A more substantial treatment of simplicity and tracial structure was provided by Ivanov in \cite{ivanov2007structurereducedamalgamatedfree}. Ueda obtains simplicity for various HNN extensions \cite{Ueda2008Remarks}, and several other substantial collections of examples arising from groups are presented in \cite{HarpePreaux2011, IvanovOmland2017, Bryder_Ivanov_Omland_2020}. Despite these results on the structure of reduced amalgamated free products and HNN extensions, the scope of the results remains limited. For instance, the literature at the moment lacks general results on stable rank 1 or strict comparison for such algebras. In this article we will address this in a streamlined manner and prove general results concerning these families of $C^*$-algebras.

We build on the emerging landscape of selfless $C^*$-algebras. This notion was introduced by Robert in \cite{robertselfless} as a unifying perspective encompassing simplicity, unique-trace, stable rank 1, and, importantly, strict comparison in $C^*$-algebras. The effectiveness and breadth of this approach was discovered in \cite{amrutam2025strictcomparisonreducedgroup} wherein several natural classes of reduced group $C^*$-algebras were proved to be selfless. This resolved Blackadar's long-standing  strict comparison problem for $C^*_r(\mathbb{F}_2)$ (\cite{Blackadar_1989}) and opened the doors to several developments that followed: the resolution of the $C^*$-algebraic Tarski problem \cite{elayavalli2025negativeresolutioncalgebraictarski}; selflessness for free semicircular $C^*$-algebras and more free products \cite{HKER, elayavalli2025remarks}; strict comparison for families of twisted group $C^*$-algebras \cite{raum2025strictcomparisontwistedgroup, flores2026purenessstablerankreduced}; $C^*$-selflessness for linear groups \cite{vigdorovich2025structuralpropertiesreducedcalgebras, ozawa2025proximalityselflessnessgroupcalgebras, vigdorovich2026selflessreducedcalgebraslinear, avni2025mixedidentitieslineargroups}; selflessness for groups with extreme-boundaries and tensor products without RD assumptions \cite{ozawa2025proximalityselflessnessgroupcalgebras}; $C^*$-selflessness for acylindrically hyperbolic groups \cite{amrutam2025strictcomparisonreducedgroup, BradfordSisto2026, ozawa2025proximalityselflessnessgroupcalgebras, yang2025extreme}; selflessness for graph products and free products without RD assumptions \cite{fmmm2025selflessfreeprod}; and selfless inclusions and selflessness for the free unitary quantum group \cite{hayes2025selfless}. Our first main result in this paper is the proof of selflessness for a general family of amalgamated free products, stated in the stronger setup of selfless inclusions.

\begin{theorem}\label{thm: main thm AFP}
    Let $(A_i,\rho_i)$ for $i=1,2$ be $C^*$-algebras which contain a common subalgebra $B$ with state-preserving expectations $E_i$. Let $E$ denote the induced expectation from $A=A_1*_BA_2$ onto $B$. Suppose that there is a unitary $a$ in the centralizer of $\rho_1$ and unitaries $b,c$ in the centralizer of $\rho_2$ such that
    \begin{enumerate}
        \item $E(a) = E(b)=E(c)=E(c^*b)=0$;
        \item $a^*Ba$ and $B$ are orthogonal in $(A_1,\rho_1)$.
    \end{enumerate}
    Then $C^*(a,b,c) \subset A_1*_BA_2$ is selfless. In particular, $A_1*_B A_2$ is selfless.
\end{theorem}

    Note that one cannot hope for amalgamated free products to be selfless in essentially full generality, as is the case in free products. Selfless $C^*$-algebras are simple by \cite{robertselfless}. Taking $B=C(X)$ and $A_1,A_2 \cong B\otimes C(X)$, $B$ will be central in $A_1*_BA_2$ and so the algebra will be neither simple nor monotracial. It is not clear if non-simplicity and existence of multiple traces are the only reasonable obstructions to selflessness in this setting, as is the case in free products. However, one may still hope that some classes of non-simple or non-monotracial amalgamated free products still are pure or have strict comparison.

    The proof of the above theorem follows the roadmap established in \cite{hayes2025selfless} using a version of Ozawa's PHP property in the $C^*$-algebraic setting. In order to fully exploit the paradoxical decompositions that naturally arise via word decompositions in amalgamated free products, we need to impose some non-triviality assumptions on the inclusion of the amalgam. Apart from natural analogues of Avitzour's conditions \cite{Avitzour}, we additionally need orthogonality relations mimicking the behavior of ``weak-malnormality'' ($H<G$ is said to be weakly malnormal if there exists $g\in G$ such that $gHg^{-1}\cap H$ is trivial). In the context of groups, by results of \cite{MO15}, it is known that amalgamated free products are acylindrically hyperbolic whenever the amalgam is weakly-malnormal, and are therefore selfless by \cite{ozawa2025proximalityselflessnessgroupcalgebras, yang2025extreme}. Hence our non-triviality conditions are natural and well motivated. A natural example where our theorem directly applies is in the case of graph products. Recall that the diameter of a simplicial graph is defined as $D(\Gamma)= \sup_{u,v\in V(\Gamma)}d(u,v)$. If the diameter is larger than 2, it is clear that the graph is \emph{irreducible}; i.e., it does not split as the join of two non trivial subgraphs. Selflessness of certain graph products of $C^*$-probability spaces over irreducible graphs was obtained in \cite{fmmm2025selflessfreeprod}. Our Theorem A allows us to obtain a very short proof of selflessness for graph products under the slightly stronger assumption of the diameter being larger than 3 (and the number of vertices being larger than 2) via an amalgamated free product decomposition. Indeed, the diameter being greater than 3 allows for the necessary weak malnormality condition to be satisfied. 

\begin{theorem}
    Let $\Gamma$ be a finite simplicial graph which has more than 2 vertices and whose diameter is greater than 3, and for each $v\in V(\Gamma)$ let $(A_v,\rho_v)$ be $C^*$-probability spaces admitting unitaries $u_v\in A_v$ in the centralizer of $\rho_v$ satisfying $\rho_v(u_v)=0$. Then the graph product $\star_{v\in \Gamma}A_v$ is selfless.
\end{theorem}

Our next main result is about selflessness for reduced HNN extensions. Motivated by the HNN extension construction in group theory, Ueda introduced the corresponding construction in the operator algebras context. We recall that the purpose of this construction is to embed the original algebra into a larger algebra where a given pair of embeddings of a fixed subalgebra are precisely unitarily conjugated. Ueda's definition is quite technical and yields a particular compression of an amalgamated free product. Here we present a fresh and rather viable approach to construct the reduced HNN extension. Our approach is inspired by the work \cite{gao2025conjugacy}. More precisely, given a $C^*$-probability space $(A, \rho)$ and a pair of subalgebras $B_1,B_2$ with expectation that are state-preservingly isomorphic via $\theta:B_1\to B_2$, we construct $\mathrm{HNN}(A,\theta, B_1,B_2)$ as a \emph{unital subalgebra} of an amalgamated free product. Using this and the methods behind Theorem \ref{thm: main thm AFP}, we are able to prove the following result.

\begin{theorem}
    Let $B_1,B_{-1}\subset (A,\rho)$ be subalgebras with expectations $E_1,E_{-1}$ and let $\theta:B_1\to B_{-1}$ be a state-preserving isomorphism. Let $a\in A\ominus B_1$ be a unitary centralizing $\rho.$ Suppose that $a^*B_1a\perp B_{-1}$. Then $C^\ast(a, w) \subset \mathrm{HNN}(A,\theta,B_1,B_{-1})$ is selfless. In particular, $\mathrm{HNN}(A,\theta,B_1,B_{-1})$ is selfless.
\end{theorem}

Before we conclude the introduction, we would like to remark on prior results around these considerations. As mentioned before the earliest results on the structure of reduced amalgamated free products of $C^*$-algebras are found in McClanahan's work, which addresses simplicity. However, the conditions imposed by McClanahan force the Avitzour unitaries to commute with the amalgam, a constraint which does not appear in our results. The next work of Ivanov \cite{ivanov2007structurereducedamalgamatedfree} concerning amalgamated free products is very relevant to our results. The hypotheses on the amalgam in Ivanov's simplicity results involve malnormality much like ours (see Corollary 4.6 \cite{ivanov2007structurereducedamalgamatedfree}). The work of Ivanov and Omland additionally discusses optimal conditions for simplicity in amalgamated free products \cite{IvanovOmland2017} (see also \cite{Bryder_Ivanov_Omland_2020} for other examples). While our work does not obtain optimal conditions for selflessness, it is very plausible that such conditions exist and we leave it to future explorations. Additionally, we point out that in line with the discussion before Theorem 4.9 in \cite{ivanov2007structurereducedamalgamatedfree}, our results give a complete characterization of $C^*$-selflessness for Baumslag--Solitar groups, namely, such a group is $C^*$-selfless iff it is $C^*$-simple. However, this result would also follow from Ozawa's result on $C^*$-selflessness of groups admitting topologically free extreme boundaries \cite{ozawa2025proximalityselflessnessgroupcalgebras}.

\subsection*{Acknowledgements}
This work was done during the Brin Mathematics Research Center workshop ``Recent Developments in Operator Algebras'' in February 2026. We thank the center for its hospitality. We are also grateful to L. Robert for his helpful comments and encouragement.  The third author was supported by the Engineering and Physical Sciences Research Council (UK), grant EP/X026647/1. 

\subsection{Open Access and Data Statement}

For the purpose of Open Access, the authors have applied a CC BY public copyright license to any Author Accepted Manuscript (AAM) version arising from this submission. Data sharing is not applicable to this article as no new data were created or analyzed in this work.

\section{Preliminaries on selfless $C^*$-algebras}

Let $B_1,B_2 \subset (A,\rho)$ be two subalgebras of a $C^*$-algebra with a trace. We say that $B_1$ and $B_2$ are \emph{orthogonal} and write $B_1\perp B_2$ if $\rho(x_1^*x_2) = 0$ whenever $x_i \in \ker(\rho)\cap B_i$ for $i=1,2.$ If $B\subset A$ is a $C^*$-subalgebra such that there is a conditional expectation $E:A\to B$, then we write $A\ominus B$ to refer to $\ker(E)$. Throughout, all tensor products are minimal and amalgamated free products are reduced.

For a $C^*$-algebra $A$ and an ultrafilter $\cU$ on a set $I$ (typically, $A$ is separable and we may take $I$ to be countable), let $A^\cU$ be the ultrapower of $A$, defined as follows.

$$A^\cU := \ell^\infty(I,A)/\{(x_n)_n : \|x_n\|\to_\cU 0\}.$$

In \cite{hayes2025selfless}, the notion of a \emph{selfless inclusion} was defined and studied. We record some pertinent facts here.

\begin{definition}
    Let $(A,\rho)$ be a $C^*$-algebra with a state $\rho.$ An inclusion $B\subset (A,\rho)$ is \emph{selfless} if there is a unitary $u\in B^\cU$ such that $C^*(A,u) \subset A^\cU$ is canonically isomorphic to $A*_r C(\T)$ and such that the free product state on $A*_r C(\T)$ is the restriction of $\rho^\cU$ to $C^*(A,u)$. We say that $(A,\rho)$ is \emph{selfless} if $A\subset (A,\rho)$ is selfless.
\end{definition}

The following is Theorem 2.5(i) of \cite{hayes2025selfless}.

\begin{proposition}
    Let $B\subset (A,\rho)$ be selfless and $B\subset C\subset A$. Then $C$ is selfless.
\end{proposition}

The following definition is inspired by Ozawa's work on selflessness \cite{ozawa2025proximalityselflessnessgroupcalgebras} and appears explicitly as Definition 3.1 in \cite{hayes2025selfless}.

\begin{definition}
    Let $B\subset (A,\rho)$ and assume $A\subset \mathbb B(\mathcal H)$ is a faithful representation. We say the inclusion $B\subset (A,\rho)$ has the PHP property if for all finite subsets $F\subset \ker \rho,$ $\varepsilon>0,$ and $n\in\N$ there exist, for $1\le i\le n$, $u_i$ unitaries in $B$ and $P_i^+\le P_i$ projections in $\mathbb B(\mathcal H)$ such that 
    \begin{enumerate}
        \item the family $\{P_i\}_{i=1}^n$ is pairwise orthogonal;
        \item for all $1\le i\le n$, $u_i(1-(P_i-P_i^+))u_i^* \le P_i^+$;
        \item for all $x\in F$ and $1\le i,j\le n$, $\|P_ixP_j\| < \varepsilon.$
    \end{enumerate}
\end{definition}

The following two statements appear as Lemma 3.2 and Theorem 3.3, respectively, in \cite{hayes2025selfless}.

\begin{lemma}\label{lem: php n is 3}
    In the above definition, it suffices to take $n=3$ and to take $F$ to be a finite subset of a densely spanning subset of $\ker \rho.$
\end{lemma}

\begin{theorem}\label{thm: php-implies-selfless}
    If $B\subset(A,\rho)$ has PHP, then $B\subset (A,\rho)$ is selfless.
\end{theorem}

\section{Proofs of main results}

\subsection{Selflessness for amalgamated free products}

\begin{theorem}\label{thm: selfless-afp}
    Let $(A_i,\rho_i)$ for $i=1,2$ be $C^*$-algebras which contain a common subalgebra $B$ with state-preserving expectations $E_i$. Let $E$ denote the induced expectation from $A=A_1*_BA_2$ onto $B$. Suppose that there is a unitary $a$ in the centralizer of $\rho_1$ and unitaries $b,c$ in the centralizer of $\rho_2$ such that
    \begin{enumerate}
        \item $E(a) = E(b)=E(c)=E(c^*b)=0$;
        \item $a^*Ba$ and $B$ are orthogonal in $(A_1,\rho_1)$.
    \end{enumerate}
    Then $C^*(a,b,c) \subset A_1*_BA_2$ is selfless. In particular, $A_1*_B A_2$ is selfless.
\end{theorem}

\begin{proof}
     We begin with some preliminaries on words in amalgamated free products. We say a product $x=x_1x_2\cdots x_n$ is a reduced word if the $x_i$ alternatively come from $A_1\ominus B$ and $A_2\ominus B$; we say $x$ has length $n$. We also consider elements of $B\ominus \C$ to be reduced words of length 0. The inner product on $L^2(A)$ is given by $\rho((w')^*w)$. We record two facts we will need in the proof of the theorem. They can be proved in a straightforward manner using the construction of amalgamated free products and induction. 

     \begin{fact}
         If $w,w'$ are reduced words of different lengths, then their inner product is 0. Furthermore, if the first term in $w$ is $x_1$ and the first term in $w'$ is $x_1'$, and if $E((x_1')^*x_1) = 0$, then $w\perp w'$. 
     \end{fact}

     For this second fact, we crucially need the hypothesis that $a^*Ba$ and $B$ are orthogonal. Indeed, when reducing the expression $a^{-1}(ac)^{-N}x(ac)^Na$, reduction can only occur when some subword is in the amalgam $B$, but then conjugating by $a$ makes this orthogonal to $B$, restoring the structure of a reduced word. The only exception is when we land in the scalars, $\C$, leading to the second case below.

     \begin{fact}\label{fact:possible-terms}
         For all reduced words $x \in A\ominus \C$, there is $N$ sufficiently large such that $a^{-1}(ac)^{-N}x(ac)^Na$ is a linear combination of reduced words of the following two forms:
         \begin{enumerate}
             \item beginning and ending with a term in $A_1\ominus B$;
            \item $(ca)^n$ for some $n\neq 0$.
         \end{enumerate}
     \end{fact}

     For a reduced word $w$, denote by $K_w$ the right $B$-module generated by words starting with $w$. For example, if $w$ ends with a term from $A_1\ominus B$ then $K_w$ consists of linear combinations of elements of the form $w\eta$ where $\eta$ is either a reduced word starting with a term from $A_2\ominus B$, or in $B$. Let $P_w$ denote the orthogonal projection onto $L^2(K_w)$.
     
      Looking to apply Lemma \ref{lem: php n is 3}, we prove the PHP condition for $n=3$ and for $F$ a finite subset of reduced words. It clearly suffices to check the PHP condition for a conjugate of $F$ by a unitary in the intersection of $C^\ast(a, b, c)$ with the centralizer of $\rho$. So, we can replace $F$ with $a^{-1}(ac)^{-N}F(ac)^Na$ for $N$ sufficiently large as to apply Fact \ref{fact:possible-terms}. For $i=1,2,3$, set $w_i = (ba)^ica(ba)^{3-i}$ and $w_i^+ = w_iba$. Set $P_i = P_{w_i}$, $P_i^+ = P_{w_i^+}$, and $u_i = w_i^+c^{-1}w_i^*$. We now verify the PHP conditions.

      For (1), we must check that $P_{w_i}\perp P_{w_j}$ for $i< j$. It suffices to check that $K_{w_i}\perp K_{w_j}$. Since $a,b$ are unitaries, it suffices to check that $(ba)^{-i}K_{w_i} \perp (ba)^{-i}K_{w_j}$. But $(ba)^{-i}K_{w_i} = K_{ca(ba)^{3-i}}$ and $(ba)^{-i}K_{w_j} = K_{(ba)^{j-i}ca(ba)^{3-j}}$. Now these are orthogonal since $E(c^*b) = 0.$

      For (2), fix $i$ and set $u=u_i$, $w=w_i$, $w^+=w_i^+$, $K = K_{w}$, and $K^+ = K_{w^+}$. We will prove both that $uK^+\subset K^+$ and that $u(L^2(A)\ominus K)\subset K^+$. First, let $\xi = w^+\eta$ be a reduced word in $K^+$. Then $\eta$ starts with a term in $A_2\ominus B$ or is in $B$. $uw^+\eta = w^+c^{-1}ba\eta$ which is also reduced as $E(c^*b) = 0$, and clearly starts with $w^+$, so $uK^+\subset K^+$. Now consider a reduced word $\xi \in L^2(A)\ominus K$. We wish to show that $u\xi = w^+c^{-1}w^*\xi \in K^+$. Let $\xi'$ be in the span of $B$ and reduced words starting with a term in $A_2\ominus B$. Then $w\xi'$ is reduced, so $w\xi' \in K$. Thus $\langle w^*\xi,\xi'\rangle=\langle \xi,w\xi'\rangle = 0$. Hence $w^*\xi$ must be in the span of reduced words starting with terms in $A_1\ominus B$. Therefore, $w^+c^{-1}w^*\xi \in K^+$.

      For (3), we will prove that for all $x\in F$ and $i,j$, we have $P_jxP_i=0.$ It suffices to show, for all $x\in F$ and all $i,j$, that $xw_i\eta \perp w_j\eta'$ for all reduced words of the form $w_i\eta$ and $w_j\eta'$. Using Fact \ref{fact:possible-terms}, we proceed. Let $x$ be a reduced word starting and ending with a term in $A_1\ominus B$; then $xw_i\eta$ is reduced and starts with a term in $A_1\ominus B$ while $w_j\eta'$ starts with a term in $A_2\ominus B$, so we get orthogonality. Otherwise, consider $x = (ca)^n$. If $n>0,$ $xw_i\eta$ is reduced and starts with $c$, while $w_j\eta'$ starts with $b$, so they are orthogonal. If $n<0$, since $E(c^*b) = 0,$ $xw_i\eta$ is a reduced word starting with $a^*$ while $w_j\eta'$ starts with a term in $A_2\ominus B$, again guaranteeing orthogonality.     
\end{proof}

\subsection{Selflessness for HNN extensions}

Let $B_1,B_{-1}\subset (A,\rho)$ be $C^*$-subalgebras with state-preserving expectations $E_1,E_{-1}$. Suppose that $\theta:B_1\to B_{-1}$ is a state-preserving *-isomorphism. Then there is a $C^\ast$-algebra $\mathrm{HNN}(A,\theta,B_1,B_{-1})$ with a faithful state extending $\rho$, containing $A$ with state-preserving expectation, and and such that $\mathrm{HNN}(A,\theta,B_1,B_{-1})$ is generated by $A$ and a unitary $w$ with $wxw^\ast = \theta(x)$ for all $x\in B_1$.

The HNN extension $\mathrm{HNN}(A,\theta,B_1,B_{-1})$ can be described as a unital subalgebra of an amalgamated free product. Let all tensors be minimal and all amalgamated free products be reduced. We define first the algebra $C = ((A\otimes A)\rtimes \Z/2\Z) *_{(B_1\otimes B_{-1})}((B_1\otimes B_{-1})\rtimes \Z/2\Z)$ where the generator $u$ of the first $\Z/2\Z$ swaps tensors and the generator $v$ of the second $\Z/2\Z$ swaps tensors via $\theta$, i.e., $v(b\otimes 1)v^\ast = 1\otimes\theta(b)$ and $v(1\otimes b)v^\ast = \theta^{-1}(b) \otimes 1$. We can now define $\mathrm{HNN}(A,\theta,B_1,B_{-1})$ as the $C^\ast$-subalgebra of $C$ generated by $A = A \otimes 1$ and $w = uv$. $\mathrm{HNN}(A,\theta,B_1,B_{-1})$ naturally has a faithful state extending $\rho$, inherited from the canonical faithful state on $C$. That there is a state-preserving expectation from $\mathrm{HNN}(A,\theta,B_1,B_{-1})$ onto $A$ follows from the existence of such an expectation on $C$.

$\mathrm{HNN}(A,\theta,B_1,B_{-1})$ is also characterized by the following property:

\begin{definition}
    An element $x = x_0w^{\varepsilon_1}x_1\cdots x_{n-1}w^{\varepsilon_n}x_n \in \mathrm{HNN}(A,\theta,B_1,B_{-1})$ with $n \ge 1$, $x_i\in A$, and $\varepsilon_i\in\{-1,1\}$ is said to be \emph{reduced} if $E_{\varepsilon_i}(x_i) = 0$ whenever $\varepsilon_i\neq \varepsilon_{i+1}$ and $1 \le i \le n-1$. By convention, elements of $A\ominus \C1$ are also considered reduced. It is easy to check that the reduced words, together with scalars, densely span the algebra.
\end{definition}

\begin{proposition}\label{prop: hnn-univ-prop}
    Let $P = \mathrm{HNN}(A,\theta,B_1,B_{-1})$. Assume that $(Q,\rho_Q)$ is a $C^\ast$-algebra equipped with a faithful state, that $\pi: A \to Q$ is a state-preserving embedding, and that has a unitary $w_Q$. Suppose,
    \begin{enumerate}
        \item $\pi(\theta(x)) = w_Q\pi(x)w_Q^\ast$ for all $x\in B_1$;
        \item for all reduced $x = x_0w^{\varepsilon_1}x_1\cdots x_{n-1}w^{\varepsilon_n}x_n \in P,$ we have
        \begin{equation*}
            \rho_Q(\pi(x_0)w_Q^{\varepsilon_1}\pi(x_1)\cdots \pi(x_{n-1})w_Q^{\varepsilon_n}\pi(x_n)) = 0.
        \end{equation*}
    \end{enumerate}
    Then there exists a unique state-preserving *-homomorphism $\Tilde{\pi}:P\to Q$ extending $\pi$ and satisfying $\Tilde{\pi}(w) = w_Q.$
\end{proposition}

To prove this, we recall the following easy combinatorial lemma, the proof of which we leave to the readers:

\begin{lemma}\label{lem: amal-free-state-zero}
    Consider the reduced amalgamated free product $P = P_1 \ast_R P_2$ with the canonical faithful state $\rho_P$. Let $x = y_0u_1y_1 \cdots y_{m-1}u_my_m$ with $m \ge 1$ and
    \begin{enumerate}
        \item $y_i \in P_1$ for all $0 \le i \le m$;
        \item $E_R(y_i) = 0$ for all $1 \le i \le m - 1$;
        \item $u_i \in P_2 \ominus R$.
    \end{enumerate}
    Then $\rho_P(x) = 0$.
\end{lemma}

\begin{proof}[Proof of Proposition \ref{prop: hnn-univ-prop}]
    Since reduced words, together with scalars, span $P$, it suffices to show all reduced words have state zero in $P$. Clearly, this holds for elements of $A \ominus \mathbb{C}1$, so let $x = x_0w^{\varepsilon_1}x_1\cdots x_{n-1}w^{\varepsilon_n}x_n \in P$ be reduced. Now, recall that $P$ is a unital subalgebra of $C = ((A\otimes A)\rtimes \Z/2\Z) *_{(B_1\otimes B_{-1})}((B_1\otimes B_{-1})\rtimes \Z/2\Z)$. Recall that the copy of $A$ contained in $P$ is $A \otimes 1$ in $C$ and $w = uv$, where $u$ is the generator of the first copy of $\Z/2\Z$ and $v$ is the generator of the second copy of $\Z/2\Z$. Let $P_1 = (A\otimes A)\rtimes \Z/2\Z$, $P_2 = (B_1\otimes B_{-1})\rtimes \Z/2\Z$, and $R = B_1\otimes B_{-1}$. It then suffices to verify the conditions in Lemma \ref{lem: amal-free-state-zero}.

    For this, we proceed by induction on $n$ to prove that $x$, after combining terms, is indeed a product of the form given in Lemma \ref{lem: amal-free-state-zero}. Furthermore, the ending term $y_m$ equals $x_n$ if $\varepsilon_n = 1$ and equals $ux_n$ if $\varepsilon_n = -1$. Indeed, when $n = 1$, we have, if $\varepsilon_1 = 1$,
    \begin{equation*}
        x = x_0uvx_1 = (x_0u)(v)(x_1)
    \end{equation*}
    which is of the desired form. And, if $\varepsilon_1 = -1$,
    \begin{equation*}
        x = x_0vux_1 = (x_0)(v)(ux_1).
    \end{equation*}

    Now, assume the result holds for $n - 1$. There are 4 cases for the inductive step:
    \begin{enumerate}
        \item $\varepsilon_{n-1} = \varepsilon_n = 1$: By induction hypothesis,
        \begin{equation*}
            x_0w^{\varepsilon_1}x_1\cdots x_{n-2}w^{\varepsilon_{n-1}}x_{n-1} = y_0u_1y_1 \cdots y_{m-1}u_my_m
        \end{equation*}
        where the latter product satisfies the conditions in Lemma \ref{lem: amal-free-state-zero} and $y_m = x_{n-1}$. Thus, $y_mu = x_{n-1}u$ is orthogonal to $R$, so,
        \begin{equation*}
            x = y_0u_1y_1 \cdots y_{m-1}u_my_mwx_n = y_0u_1y_1 \cdots y_{m-1}u_m(y_mu)(v)(x_n)
        \end{equation*}
        is as required.
        \item $\varepsilon_{n-1} = 1$ and $\varepsilon_n = -1$: Again, by induction hypothesis,
        \begin{equation*}
            x_0w^{\varepsilon_1}x_1\cdots x_{n-2}w^{\varepsilon_{n-1}}x_{n-1} = y_0u_1y_1 \cdots y_{m-1}u_my_m
        \end{equation*}
        where the latter product satisfies the conditions in Lemma \ref{lem: amal-free-state-zero} and $y_m = x_{n-1}$. Furthermore, by definition of reduced words, $x_{n-1} \perp B_1$, so $y_m = x_{n-1} \otimes 1$ is orthogonal to $R = B_1 \otimes B_{-1}$. Thus,
        \begin{equation*}
            x = y_0u_1y_1 \cdots y_{m-1}u_my_mw^{-1}x_n = y_0u_1y_1 \cdots y_{m-1}u_m(y_m)(v)(ux_n)
        \end{equation*}
        is as required.
        \item $\varepsilon_{n-1} = -1$ and $\varepsilon_n = 1$: By induction hypothesis,
        \begin{equation*}
            x_0w^{\varepsilon_1}x_1\cdots x_{n-2}w^{\varepsilon_{n-1}}x_{n-1} = y_0u_1y_1 \cdots y_{m-1}u_my_m
        \end{equation*}
        where the latter product satisfies the conditions in Lemma \ref{lem: amal-free-state-zero} and $y_m = ux_{n-1}$. Furthermore, by definition of reduced words, $x_{n-1} \perp B_{-1}$, so $y_mu = ux_{n-1}u = 1 \otimes x_{n-1}$ is orthogonal to $R = B_1 \otimes B_{-1}$. Thus,
        \begin{equation*}
            x = y_0u_1y_1 \cdots y_{m-1}u_my_mwx_n = y_0u_1y_1 \cdots y_{m-1}u_m(y_mu)(v)(x_n)
        \end{equation*}
        is as required.
        \item $\varepsilon_{n-1} = \varepsilon_n = -1$: By induction hypothesis,
        \begin{equation*}
            x_0w^{\varepsilon_1}x_1\cdots x_{n-2}w^{\varepsilon_{n-1}}x_{n-1} = y_0u_1y_1 \cdots y_{m-1}u_my_m
        \end{equation*}
        where the latter product satisfies the conditions in Lemma \ref{lem: amal-free-state-zero} and $y_m = ux_{n-1}$. Thus, $y_m$ is orthogonal to $R$, so,
        \begin{equation*}
            x = y_0u_1y_1 \cdots y_{m-1}u_my_mw^{-1}x_n = y_0u_1y_1 \cdots y_{m-1}u_m(y_m)(v)(ux_n)
        \end{equation*}
        is as required.
    \end{enumerate}
\end{proof}

\begin{remark}
    In the definition of the HNN extension, we can also replace the tensor products with free products; i.e., we can define $\mathrm{HNN}(A, \theta, B_1, B_{-1})$ as a unital subalgebra of $C = ((A \ast A) \rtimes \Z/2\Z) \ast_{(B_1 \ast B_{-1})} ((B_1 \ast B_{-1}) \rtimes \Z/2\Z)$. The proof of Proposition \ref{prop: hnn-univ-prop} carries through nearly verbatim, so in fact the HNN extensions defined in these two a priori different ways are state-preservingly isomorphic.
\end{remark}

\begin{theorem}
    Let $B_1,B_{-1}\subset (A,\rho)$ be subalgebras with expectations $E_1,E_{-1}$ and let $\theta:B_1\to B_{-1}$ be a state-preserving isomorphism. Let $a\in A\ominus B_1$ be a unitary centralizing $\rho.$ Suppose that $a^*B_1a\perp B_{-1}$. Then $C^\ast(a, w) \subset \mathrm{HNN}(A,\theta,B_1,B_{-1})$ is selfless. In particular, $\mathrm{HNN}(A,\theta,B_1,B_{-1})$ is selfless.
\end{theorem}

\begin{proof}
    Set $C = \mathrm{HNN}(A,\theta,B_1,B_{-1})$. Let $w \in C$ be the unitary such that $wxw^* = \theta(x)\in B_{-1}$ for all $x\in B_1.$ Recall from the above discussion that reduced words are of the form $x_0w^{\ee_1}x_1 \cdots x_{n-1}w^{\ee_{n}}x_n$ where $n\geq 1$, $\ee_i\in\{1,-1\},$ and whenever $1\leq i\leq n-1$ and $\ee_i\neq \ee_{i+1}$, $E_{\ee_i}(x_i)=0$; elements of $A\ominus \C1$ are also considered reduced words. Also, all reduced words have trace 0. We note that if two reduced words have different sequences of $w$'s and $w^{-1}$'s then they are orthogonal; if two reduced words start with $x_1w$ and $x_2w$ respectively and $E_{-1}(x_2^*x_1)=0$, then the two reduced words are orthogonal; if two reduced words start with $x_1w^{-1}$ and $x_2w^{-1}$ respectively and $E_{1}(x_2^*x_1)=0$, then the two reduced words are orthogonal. (These orthogonality facts can be seen using a combinatorial argument that shows the relevant products of words can all be written as linear combinations of reduced words; see also the $L^2$-space structure of the HNN extension as described in \cite[Section~3]{fima2012hnn}.)

    We will prove the PHP property. To that end, let $F\subset C\ominus \C1$ be a finite set of state zero elements; without loss of generality, we may assume $F$ consists of reduced words. Since $a^*B_1a\perp B_{-1}$, it follows that for $N$ sufficiently large, for all $x\in F$ $(aw)^Nx(aw)^{-N}$ is a linear combination of reduced words of the following forms:
    \begin{enumerate}
        \item starting with $aw$ and ending with $w^{-1}a^{-1}$;
        \item $(aw)^n$ for some $n\neq 0.$
    \end{enumerate}
    So without loss of generality, we may assume that $F$ consists of reduced words of the above two forms.

    For $i=1,2,3,$ we set $v_i = (w^{-2}a)^iw^{-1}a(w^{-2}a)^{3-i}$, $v_i^+ = v_i w^{-2}a$, and $u_i = v_i^+(aw)^{-2}w(aw)^2v_i^{-1}$. We note that $v_i$ and $v_i^+$ are reduced. For $u_i$, while it may not be reduced, it is a sum of two reduced words. Indeed,
    \begin{equation}\label{eqn: reduction computation}
    \begin{split}
        &u_i\\
        = &v_i^+w^{-1}a^\ast w^{-1}a^\ast wawawv_i^{-1}\\
        = &v_i^+w^{-1}a^\ast w^{-1}(a^\ast - E_{-1}(a^\ast))wawawv_i^{-1} + v_i^+w^{-1}a^\ast\theta^{-1}(E_{-1}(a^\ast))awawv_i^{-1}
    \end{split}
    \end{equation}
    and for the second term, we note that, as $a \in A \ominus B_1$, $\rho(\theta^{-1}(E_{-1}(a))) = 0$. Hence, as $a^\ast B_1a \perp B_{-1}$, we have $a^\ast\theta^{-1}(E_{-1}(a^\ast))a \perp B_{-1}$. From this, it follows that both terms in the above are indeed reduced.
    
    We also set, for a reduced word $y,$ $K_y$ to be the span of reduced words starting with $y$, and $P_y$ the orthogonal projection onto $L^2(K_y)$. We now verify the three conditions for PHP, using the unitaries $u_i$ and the projections $P_{v_i^+} \leq P_{v_i}$.

    For (1), it suffices to check that for $\eta,\eta'$ such that $v_i\eta$ and $v_j\eta'$ are reduced that $v_i\eta$ and $v_j\eta'$ are orthogonal. Without loss of generality, $i>j,$ and since orthogonality is preserved by unitary transformations, it suffices to show that $w^{-1}a(w^{-2}a)^{i-j-1}w^{-1}a(w^{-2}a)^{3-i}\eta$ and $a(w^{-2}a)^{3-j}\eta'$ are orthogonal. But $E_1(a)=0,$ so these are indeed orthogonal.

    For (2), fix $i$ and set $v=v_i,$ $v^+ = v_i^+$, $u = u_i$, $K = K_{v}$, and $K^+ = K_{v^+}$. We must check both that $uK^+\subset K^+$ and $u(L^2(C)\ominus K) \subset K^+$. First, consider a reduced $v^+\eta = v w^{-2}a \eta \in K^+$. Since this is reduced, either the first power of $w$ in $\eta$ is negative or the leading term $x_0\in A$ satisfies $E_{-1}(ax_0) = 0$. Then $uv^+\eta = v^+ (aw)^{-2}wawaw^{-1}a\eta$. By the same computation as in Equation (\ref{eqn: reduction computation}) and using $a \in A \ominus B_1$ so $waw^{-1}$ is reduced, we see that $uv^+\eta$ is a sum of two reduced words starting with $v^+$, and so is in $K^+$. Now let $\xi \in L^2(C)\ominus K$. Let $\xi'$ be in the span of $A$ and reduced words starting with $Aw^{-1}$. Then $v\xi'$ is reduced and starts with $v$. Thus, $\langle \xi', v^*\xi\rangle = \langle v\xi',\xi\rangle = 0$. Therefore it must be that $v^*\xi = v^{-1}\xi$ is a linear combination of reduced words starting with $Aw.$ Now, $u\xi = v^+(aw)^{-2}w(aw)^2v^{-1}\xi$. So, again by the same computation as in Equation (\ref{eqn: reduction computation}) and using that $v^{-1}\xi$ is a linear combination of reduced words starting with $Aw,$ we see that $u\xi$ is a linear combination of reduced words starting with $v^+$, and so is in $K^+$.

    Lastly, for (3) it suffices to show that $xv_i\eta$ is orthogonal to $v_j\eta'$ for all $x\in F$, $i,j$, and $\eta,\eta'$ such that $v_i\eta$ and $v_j\eta'$ are reduced. If $x$ starts with $aw$ and ends with $w^{-1}a^{-1}$, then $xv_i\eta$ is reduced and its first power of $w$ is positive, while in $v_j\eta'$ the first power of $w$ is negative, so they are orthogonal. If $x = (aw)^n$ for $n<0,$ again $xv_i\eta$ is reduced. Left multiplying both $xv_i\eta$ and $v_j\eta'$ by $w$ gives on the one hand a reduced word starting with $a^{-1}w^{-1}$ and on the other hand a reduced word starting with $w^{-1}$. Since $E_{1}(a) = 0$, these are orthogonal. If $x = (aw)^n$ for $n\ge2$, then $xv_i\eta$ reduces to $(aw)^{n-1}aw^{-1}a(w^{-2}a)^{i-1}w^{-1}a(w^{-2}a)^{3-i}\eta$, which is reduced since $E_1(a) = 0$. This word's first power of $w$ is positive and so is orthogonal to $v_j\eta'$. Lastly, if $x = aw$ then $xv_i\eta$ reduces to $aw^{-1}a(w^{-2}a)^{i-1}w^{-1}a(w^{-2}a)^{3-i}\eta$ which is again orthogonal to $v_j\eta'$ since $E_1(a)=0.$

    We now deduce that the inclusion $C^*(a,w)\subset C$ is selfless, so that in particular the HNN extension $C$ is selfless.
\end{proof}

\begin{remark}
    We can form a ``graph of $C^*$-algebras'' in the same sense as a graph of groups \cite{Ser80} via iterated amalgamated free products and HNN extensions, and also obtain selflessness in such settings. We omit commenting more on this to keep the article concise. 
\end{remark}
    
\bibliographystyle{alpha}
	\bibliography{references}

@article {MO15,
    AUTHOR = {Minasyan, Ashot and Osin, D. V.},
     TITLE = {Acylindrical hyperbolicity of groups acting on trees},
   JOURNAL = {Math. Ann.},
    VOLUME = {362},
      YEAR = {2015},
     PAGES = {1055--1105},
}

@article{Popa1993,
  author  = {Popa, Sorin},
  title   = {Markov traces on universal {J}ones algebras and subfactors of finite index},
  fjournal = {Inventiones mathematicae},
  journal={Invent. Math.},
  year    = {1993},
  volume  = {111},
  number  = {2},
  pages   = {375--405},
  doi     = {10.1007/BF01231293}}

@article {IPP08,
    AUTHOR = {Ioana, Adrian and Peterson, Jesse and Popa, Sorin},
     TITLE = {Amalgamated free products of weakly rigid factors and
              calculation of their symmetry groups},
   JOURNAL = {Acta Math.},
  FJOURNAL = {Acta Mathematica},
    VOLUME = {200},
      YEAR = {2008},
    NUMBER = {1},
     PAGES = {85--153},
      ISSN = {0001-5962},
   MRCLASS = {46L10 (20E06 22D25 37A15 37A20 46L09)},
  MRNUMBER = {2386109},
MRREVIEWER = {Alain Valette},
       DOI = {10.1007/s11511-008-0024-5},
       URL = {https://doi-org.proxy.library.vanderbilt.edu/10.1007/s11511-008-0024-5},
}

@article{higman1961subgroups,
  title={Subgroups of finitely presented groups},
  author={Higman, Graham},
  fjournal={Proceedings of the Royal Society of London. Series A. Mathematical and Physical Sciences},
  journal={Proc. Roy. Soc. Lond. Ser. A. Math. Phys. Sci.},
  volume={262},
  number={1311},
  pages={455--475},
  year={1961},
  publisher={The Royal Society London}
}

@book{speicher1998combinatorial,
  title={Combinatorial Theory of the Free Product with Amalgamation and Operator-Valued Free Probability Theory},
  author={Speicher, Roland},
  fseries={Memoirs of the American Mathematical Society},
  series={Mem. Amer. Math. Soc.},
  volume={132},
  number={627},
  year={1998},
  publisher={American Mathematical Society},
  address={Providence, R.I.},
  isbn={978-0-8218-0693-7}
}

@article {KhinRX, AUTHOR = {Ricard, \'{E}ric and Xu, Quanhua}, TITLE = {Khintchine type inequalities for reduced free products and applications}, JOURNAL = {J. Reine Angew. Math.}, FJOURNAL = {Journal f\"{u}r die Reine und Angewandte Mathematik. [Crelle's Journal]}, VOLUME = {599}, YEAR = {2006}, PAGES = {27--59}, ISSN = {0075-4102}, MRCLASS = {46L09 (46L05 46L10)}, MRNUMBER = {2279097}, MRREVIEWER = {Maria Grazia Viola}, DOI = {10.1515/CRELLE.2006.077}, URL = {https://doi.org/10.1515/CRELLE.2006.077}, }

@article {stablerankHaagerup,
    AUTHOR = {Dykema, Kenneth J. and Haagerup, Uffe and R{\o}rdam, Mikael},
     TITLE = {The stable rank of some free product {$C^\ast$}-algebras},
   JOURNAL = {Duke Math. J.},
  FJOURNAL = {Duke Mathematical Journal},
    VOLUME = {90},
      YEAR = {1997},
    NUMBER = {1},
     PAGES = {95--121},
      ISSN = {0012-7094,1547-7398},
   MRCLASS = {46L05 (19B10 19K56 46L35 46L80)},
  MRNUMBER = {1478545},
MRREVIEWER = {Gustavo\ Corach},
       DOI = {10.1215/S0012-7094-97-09004-9},
       URL = {https://doi.org/10.1215/S0012-7094-97-09004-9},
}

@article{ivanov2007structurereducedamalgamatedfree,
  title={On the structure of some reduced amalgamated free product {$C^*$}-algebras},
  author={Ivanov, Nikolay A},
  fjournal={International Journal of Mathematics},
  journal={Internat. J. Math.},
  volume={22},
  number={02},
  pages={281--306},
  year={2011},
  publisher={World Scientific}
}

@article{IvanovOmland2017,
  author  = {Ivanov, Nikolay A. and Omland, Tron},
  title   = {{$C^*$}-simplicity of free products with amalgamation and radical classes of groups},
  fjournal = {Journal of Functional Analysis},
  journal={J. Funct. Anal.},
  year    = {2017},
  volume  = {272},
  number  = {9},
  pages   = {3712--3741},
  doi     = {10.1016/j.jfa.2017.01.011},
  url     = {https://doi.org/10.48550/arXiv.1605.06395}
}

@article{avni2025mixedidentitieslineargroups,
      title={Mixed identities in linear groups -- effective version}, 
      author={Nir Avni and Tsachik Gelander},
      year={2025},
      eprint={2510.03492},
      archivePrefix={arXiv},
      primaryClass={math.GR},
      url={https://arxiv.org/abs/2510.03492}, 
      journal={arXiv:2510.03492},
}

@article{vigdorovich2026selflessreducedcalgebraslinear,
      title={Selfless reduced {$C^*$}-algebras of linear groups}, 
      author={Itamar Vigdorovich},
      year={2026},
      eprint={2602.10616},
      archivePrefix={arXiv},
      primaryClass={math.OA},
      url={https://arxiv.org/abs/2602.10616}, 
      journal={arXiv:2602.10616},
}

@article{flores2026purenessstablerankreduced,
      title={Pureness and stable rank one for reduced twisted group {$C^\ast$}-algebras of certain group extensions}, 
      author={Flores, Felipe and Klisse, Mario and {\'O} Cobhthaigh, M{\'i}che{\'a}l and Pagliero, Matteo},
      year={2026},
      eprint={2601.19758},
      archivePrefix={arXiv},
      primaryClass={math.OA},
      url={https://arxiv.org/abs/2601.19758}, 
      journal={arXiv:2601.19758},
}

@inproceedings{Blackadar_1989, 
place={Cambridge}, 
series={London Math. Soc. Lecture Note Ser.},
fseries={London Mathematical Society Lecture Note Series}, 
title={Comparison theory for simple {$C^\ast$}-algebras}, 
booktitle={Operator Algebras and Applications}, 
publisher={Cambridge University Press}, 
author={Blackadar, Bruce}, 
editor={Evans, David E. and Takesaki, Masamichi}, 
year={1989}, 
pages={21--54}, 
collection={London Mathematical Society Lecture Note Series}}

@article{Bryder_Ivanov_Omland_2020,
  author    = {Bryder, Rasmus Sylvester and Ivanov, Nikolay A. and Omland, Tron},
  title     = {{$C^*$}-simplicity of {HNN~extensions} and groups~acting~on~trees},
  fjournal   = {Annales de l'Institut Fourier},
  journal={Ann. Inst. Fourier},
  volume    = {70},
  number    = {4},
  pages     = {1497--1543},
  year      = {2020},
  doi       = {10.5802/aif.3378},
  publisher = {Association des Annales de l'Institut Fourier},
  url       = {https://aif.centre-mersenne.org}
}

@article{HarpePreaux2011,
  title     = {{$C^*$}-simple groups: amalgamated free products, {HNN} extensions, and fundamental groups of 3-manifolds},
  author    = {de la Harpe, Pierre and Pr{\'e}aux, Jean-Philippe},
  fjournal   = {Journal of Topology and Analysis},
  journal={J. Topol. Anal.},
  volume    = {3},
  number    = {2},
  pages     = {125--158},
  year      = {2011},
  publisher = {World Scientific},
  doi       = {10.1142/S179352531100054X},
  eprint    = {0909.3528},
  archivePrefix = {arXiv},
  primaryClass  = {math.GR}
}

@article{McClanahan1994,
  author    = {McClanahan, Kevin},
  title     = {Simplicity of Reduced Amalgamated Products of {$C^*$}-Algebras},
  fjournal   = {Canadian Journal of Mathematics},
  journal={Canad. J. Math.},
  volume    = {46},
  number    = {4},
  pages     = {793--814},
  year      = {1994},
  doi       = {10.4153/CJM-1994-041-x},
  publisher = {Cambridge University Press}
}

@article {Dykemasimplicitystablerankonefree,
    AUTHOR = {Dykema, Kenneth J.},
     TITLE = {Simplicity and the stable rank of some free product
              {$C^*$}-algebras},
   JOURNAL = {Trans. Amer. Math. Soc.},
  FJOURNAL = {Transactions of the American Mathematical Society},
    VOLUME = {351},
      YEAR = {1999},
    NUMBER = {1},
     PAGES = {1--40},
      ISSN = {0002-9947,1088-6850},
   MRCLASS = {46L05 (19B10 19K56 46L35 46L80)},
  MRNUMBER = {1473439},
MRREVIEWER = {Gustavo\ Corach},
       DOI = {10.1090/S0002-9947-99-02180-7},
       URL = {https://doi.org/10.1090/S0002-9947-99-02180-7},
}

@article{dykema1998projections,
  title={Projections in free product {$C^\ast$}-algebras},
  author={Dykema, Kenneth J. and R{\o}rdam, Mikael},
  journal={Geom. Funct. Anal.},
  fjournal={Geometric \& Functional Analysis GAFA},
  volume={8},
  number={1},
  pages={1--16},
  year={1998},
  publisher={Springer}
}

@article{Ueda2008Remarks,
  author    = {Ueda, Yoshimichi},
  title     = {Remarks on {HNN} extensions in operator algebras},
  fjournal   = {Illinois Journal of Mathematics},
  journal={Illinois J. Math.},
  volume    = {52},
  number    = {3},
  pages     = {705--725},
  year      = {2008},
  doi       = {10.1215/ijm/1254143997}}

@article{Ueda2005HNN,
  author    = {Ueda, Yoshimichi},
  title     = {{HNN} extensions of von {N}eumann algebras},
  fjournal   = {Journal of Functional Analysis},
  journal={J. Funct. Anal.},
  volume    = {225},
  number    = {2},
  pages     = {383--426},
  year      = {2005},
  doi       = {10.1016/j.jfa.2005.01.004},
  url       = {https://doi.org/10.1016/j.jfa.2005.01.004}
}

@article{exoticCIKE,
	author = {Chifan, Ionu{\c t} and Ioana, Adrian and Kunnawalkam Elayavalli, Srivatsav},
	date = {2023/06/17},
	doi = {10.1007/s00039-023-00649-4},
	id = {Chifan2023},
	isbn = {1420-8970},
	fjournal = {Geometric and Functional Analysis},
    journal={Geom. Funct. Anal.},
	title = {An exotic {II}$_1$ factor without property {G}amma},
	url = {https://doi.org/10.1007/s00039-023-00649-4},
	year = {2023},
    volume={33},
  number={5},
  pages={1243--1265},}

@article {CartanAFP,
    AUTHOR = {Ioana, Adrian},
     TITLE = {Cartan subalgebras of amalgamated free product {${\rm II}_1$}
              factors},
      NOTE = {With an appendix by Ioana and Stefaan Vaes},
   JOURNAL = {Ann. Sci. \'{E}c. Norm. Sup\'{e}r. (4)},
  FJOURNAL = {Annales Scientifiques de l'\'{E}cole Normale Sup\'{e}rieure. Quatri\`eme
              S\'{e}rie},
    VOLUME = {48},
      YEAR = {2015},
    NUMBER = {1},
     PAGES = {71--130},
      ISSN = {0012-9593},
   MRCLASS = {46L36 (28D15 37A20 46L10 60B15)},
  MRNUMBER = {3335839},
MRREVIEWER = {Sven Raum},
       DOI = {10.24033/asens.2239},
       URL = {https://doi.org/10.24033/asens.2239},
}

@article {mei,
    AUTHOR = {Mei, Tao and Ricard, \'{E}ric},
     TITLE = {Free {H}ilbert transforms},
   JOURNAL = {Duke Math. J.},
  FJOURNAL = {Duke Mathematical Journal},
    VOLUME = {166},
      YEAR = {2017},
    NUMBER = {11},
     PAGES = {2153--2182},
      ISSN = {0012-7094,1547-7398},
   MRCLASS = {46L07 (46L52 46L54)},
  MRNUMBER = {3694567},
MRREVIEWER = {Daniele\ Puglisi},
       DOI = {10.1215/00127094-2017-0007},
       URL = {https://doi.org/10.1215/00127094-2017-0007},
}

@incollection{voiculescu1985symmetries,
  author    = {Voiculescu, Dan},
  title     = {Symmetries of some reduced free product {$C^*$}-algebras},
  booktitle = {Operator Algebras and Their Connections with Topology and Ergodic Theory},
  series    = {Lecture Notes in Math.},
  volume    = {1132},
  pages     = {556--588},
  year      = {1985},
  publisher = {Springer-Verlag},
  address   = {Berlin, Heidelberg},
  doi       = {10.1007/BFb0074909}
}

@book{voiculescu1992free,
  title={Free Random Variables: A Noncommutative Probability Approach to Free Products with Applications to Random Matrices, Operator Algebras, and Harmonic Analysis on Free Groups},
  author={Voiculescu, Dan-Virgil and Dykema, Kenneth J. and Nica, Alexandru},
  series={CRM Monogr. Ser.},
  volume={1},
  year={1992},
  publisher={Amer. Math. Soc.},
  address={Providence, RI},
  isbn={978-0821869994}
}

@article{HNNpaper,
author = {Higman, Graham and Neumann, B. H. and Neumann, Hanna},
title = {Embedding Theorems for Groups},
fjournal = {Journal of the London Mathematical Society},
journal={J. Lond. Math. Soc.},
volume = {s1-24},
number = {4},
pages = {247-254},
doi = {https://doi.org/10.1112/jlms/s1-24.4.247},
url = {https://londmathsoc.onlinelibrary.wiley.com/doi/abs/10.1112/jlms/s1-24.4.247},
eprint = {https://londmathsoc.onlinelibrary.wiley.com/doi/pdf/10.1112/jlms/s1-24.4.247},
year = {1949}
}

@article{elayavalli2025remarks,
  title         = {Some remarks on decay in countable groups and amalgamated free products}, 
  author        = {Kunnawalkam Elayavalli, Srivatsav  and Gregory Patchell and Lizzy Teryoshin},
  year          = {2025},
  journal        = {arXiv:2509.08754}}

@article{BradfordSisto2026,
  author  = {Bradford, Henry and Sisto, Alessandro},
  title   = {Non-solutions to mixed equations in acylindrically hyperbolic groups coming from random walks},
  fjournal = {Archiv der Mathematik},
  journal={Arch. Math.},
  year    = {2026},
  volume  = {126},
  number  = {3},
  pages   = {223--234},
  doi     = {10.1007/s00013-026-02223-4}}

@article{yang2025extreme,
  title         = {An extreme boundary of acylindrically hyperbolic groups},
  author        = {Wenyuan Yang},
  year          = {2025},
  eprint        = {2511.16400},
  archivePrefix = {arXiv},
  primaryClass  = {math.GR},
  journal={arXiv:2511.16400},}

@article{hayes2025selfless,
  title={Selfless Inclusions of {$C^\ast$}-Algebras},
  author={Hayes, Ben and Kunnawalkam Elayavalli, Srivatsav  and Patchell, Gregory and Robert, Leonel},
  journal={arXiv:2510.13398},
  year={2025}
}

@book {Ser80,
    AUTHOR = {Serre, Jean-Pierre},
     TITLE = {Trees},
      NOTE = {Translated from the French by John Stillwell},
 PUBLISHER = {Springer-Verlag, Berlin-New York},
      YEAR = {1980},
     PAGES = {ix+142},
      ISBN = {3-540-10103-9},
   MRCLASS = {20H10 (05C05 22E50)},
  MRNUMBER = {607504},
}

@article{ozawa2025proximalityselflessnessgroupcalgebras,
      title={Proximality and selflessness for group {$C^*$}-algebras}, 
      author={Narutaka Ozawa},
      year={2025},
      eprint={2508.07938},
      archivePrefix={arXiv},
      primaryClass={math.OA},
      url={https://arxiv.org/abs/2508.07938},
journal={arXiv:2508.07938},
}

@article{raum2025strictcomparisontwistedgroup,
      title={Strict comparison for twisted group {$C^*$}-algebras}, 
      author={Sven Raum and Hannes Thiel and Eduard Vilalta},
      journal={arXiv:2505.18569},
      year={2025},
      eprint={2505.18569},
      archivePrefix={arXiv},
      primaryClass={math.OA},
      url={https://arxiv.org/abs/2505.18569}, 
}

@article{elayavalli2025negativeresolutioncalgebraictarski,
      title={Negative resolution to the {$C^*$}-algebraic {Tarski} problem}, 
      author={Kunnawalkam Elayavalli, Srivatsav and  Schafhauser, Christopher},
      year={2025},
      journal={arXiv:2503.10505}
}

@article {robertselfless,
    AUTHOR = {Robert, Leonel},
     TITLE = {Selfless {$C^*$}-algebras},
   JOURNAL = {Adv. Math.},
  FJOURNAL = {Advances in Mathematics},
    VOLUME = {478},
      YEAR = {2025},
     PAGES = {Paper No. 110409, 28},
      ISSN = {0001-8708,1090-2082},
   MRCLASS = {46L05 (46L35 46L54)},
  MRNUMBER = {4924062},
       DOI = {10.1016/j.aim.2025.110409},
       URL = {https://doi.org/10.1016/j.aim.2025.110409},
}

@article{vigdorovich2025structuralpropertiesreducedcalgebras,
      title={Structural properties of reduced {$C^*$}-algebras associated with higher-rank lattices}, 
      author={Itamar Vigdorovich},
      year={2025},
      journal={arXiv:2503.12737}
}

@article {Avitzour,
    AUTHOR = {Avitzour, Daniel},
     TITLE = {Free products of {$C\sp{\ast} $}-algebras},
   JOURNAL = {Trans. Amer. Math. Soc.},
  FJOURNAL = {Transactions of the American Mathematical Society},
    VOLUME = {271},
      YEAR = {1982},
    NUMBER = {2},
     PAGES = {423--435},
      ISSN = {0002-9947},
   MRCLASS = {46L05},
  MRNUMBER = {654842},
MRREVIEWER = {William Paschke},
       DOI = {10.2307/1998890},
       URL = {https://doi.org/10.2307/1998890},
}

@article{amrutam2025strictcomparisonreducedgroup,
  title={Strict comparison in reduced group {$C^*$}-algebras},
  author={Amrutam, Tattwamasi and Gao, David and Kunnawalkam Elayavalli, Srivatsav and Patchell, Gregory},
  journal={Invent. Math.},
  volume={242},
  number={3},
  pages={639--657},
  year={2025},
  publisher={Springer}
}

@article{HKER,
      title={Selfless reduced free product {$C^*$}-algebras}, 
      author={Hayes,Ben and  Kunnawalkam Elayavalli, Srivatsav and Robert, Leonel},
      year={2025},
      journal={arXiv:2505.13265},
      eprint={2505.13265},
      archivePrefix={arXiv},
      primaryClass={math.OA},
      url={https://arxiv.org/abs/2505.13265}, 
}

@article {fmmm2025selflessfreeprod,
    AUTHOR = {Flores, Felipe and Klisse, Mario and {\'O} Cobhthaigh, M{\'i}che{\'a}l and Pagliero, Matteo},
     TITLE = {Selfless reduced free products and graph products of {$C^*$}-algebras},
   journal={arXiv:2510.24675},
  year={2025}
}

@article{fima2012hnn,
  title={{HNN} extensions and unique group measure space decomposition of {II$_1$} factors},
  author={Fima, Pierre and Vaes, Stefaan},
  fjournal={Transactions of the American Mathematical Society},
  journal={Trans. Amer. Math. Soc.},
  volume={364},
  number={5},
  pages={2601--2617},
  year={2012},
  publisher={American Mathematical Society},
  doi={10.1090/S0002-9947-2011-05511-7}
}

@article{gao2025conjugacy,
  title={On conjugacy and perturbation of subalgebras},
  author={Gao, David and Kunnawalkam Elayavalli, Srivatsav and Patchell, Gregory and Tan, Hui},
  fjournal={Journal of Noncommutative Geometry},
  journal={J. Noncommut. Geom.},
  year={2025},
  publisher={EMS Press},
  doi={10.4171/JNCG/627}
}

\end{document}